\documentclass{article}

 \usepackage[margin=2.5cm]{geometry}
\usepackage{amsmath,amsthm,amssymb,enumerate,epsfig,color,graphicx,epstopdf}

\usepackage{array}
\newcolumntype{C}[1]{>{\centering\arraybackslash$}p{#1}<{$}}

\usepackage[colorlinks=true,linkcolor=blue,citecolor=blue,draft=false]{hyperref}
\usepackage[capitalise,nameinlink,noabbrev]{cleveref}
\crefname{subsection}{Subsection}{Subsections}

\def\<{\langle}
\def\>{\rangle}

\newtheorem{theorem}{Theorem}[section]
\newtheorem{lemma}[theorem]{Lemma}
\newtheorem{proposition}[theorem]{Proposition}
\newtheorem{corollary}[theorem]{Corollary}
\newtheorem{remark}[theorem]{Remark}
\newtheorem{definition}[theorem]{Definition}

\newtheorem{example}[theorem]{Example}

\title{Polynomial braid combing}
\author{Juan Gonz\'alez-Meneses and Marithania Silvero\footnote{Both authors partially supported by the Spanish research projects MTM2013-44233-P, MTM2016-76453-C2-1-P and FEDER.}}
\date{December, 2017}

\begin{document}

\maketitle


\begin{abstract}
Braid combing is a procedure defined by Emil Artin to solve the word problem in braid groups for the first time. It is well-known to have exponential complexity. In this paper, we use the theory of straight line programs to give a polynomial algorithm which performs braid combing. This procedure can be applied to braids on surfaces, providing the first algorithm (to our knowledge) which solves the word problem for braid groups on surfaces with boundary in polynomial time and space.

In the case of surfaces without boundary, braid combing needs to use a section from the fundamental group of the surface to the braid group. Such a section was shown to exist by Gon\c calves and Guaschi, who also gave a geometric description. We propose an algebraically simpler section, which we describe explicitly in terms of generators of the braid group, and we show why the above procedure to comb braids in polynomial time does not work in this case.
\end{abstract}

\section{Introduction}

Braid groups can be seen as the fundamental group of the configuration space of $n$ distinct points in a closed disc $\mathbb D$. If the points are unordered, the fundamental group is called the {\it full} braid group (or just the braid group) with $n$ strands, and denoted $B_n$. If the points are ordered, the obtained group is a finite index subgroup of $B_n$, called the {\it pure} braid group with $n$ strands, $P_n$.

If one replaces the closed disc $\mathbb D$ with any connected surface $S$, one obtains the full braid group $B_n(S)$ and the pure braid group $P_n(S)$ with $n$ strands on $S$.

Emil Artin~\cite{Artin} solved the word problem in braid groups (on the disc) for the first time. Actually, he solved the word problem in $P_n$, and then used that $B_n$ is a finite extension of $P_n$ by the symmetric group $\Sigma_n$. The way in which he solved the word problem in $P_n$ is known as {\it braid combing}. Artin showed that $P_n$ can be seen as an iterated semi-direct product of free groups:
$$
      P_n=((\cdots ((\mathbb F_2 \ltimes \mathbb F_3)\ltimes \mathbb F_4)\ltimes \cdots \mathbb F_{n-2})\ltimes \mathbb F_{n-1}.
$$
The braid combing consists on computing the normal form of a pure braid with respect to the above semi-direct decomposition. As the word problem in a free group of finite rank is well-known, this solves the word problem in $P_n$.

But there is a big issue with braid combing: It is an exponential procedure. If we start with a word of length $m$ in the standard generators of $P_n$, the length of the combed braid may be exponential in $m$, when written in terms of the generators of the free groups. An explicit example is given in \cref{subsectionexponencial}. Artin was of course aware of this; In the very last paragraph of~\cite{Artin}, in which he talks about braid combing, he says:

{\it ``Although it has been proved that every braid can be deformed into a similar normal form the writer is convinced that any attempt to carry this out on a living person would only lead to violent protests and discrimination against mathematics. He would therefore discourage such an experiment"}.

It this paper we use the theory of {\it straight line programs} (a compressed way to store a word as a set of instructions to create it), to perform braid combing in polynomial time and space. This means that we give an algorithm which, given a word $w$ of length $m$ in the standard generators of $P_n$, computes $n-1$ compressed words, each one representing a factor of the combed braid associated to $w$, in polynomial time and space with respect to $m$.

Furthermore, given two pure braids, one can compare the compressed words associated to each of them in polynomial time. Hence, this procedure gives a polynomial solution of the word problem in pure braid groups, using braid combing.

This result, in the case of classical braids, does not improve the existing algorithms, as there are quadratic solutions to the word problem in braid groups of the disc. But it happens that the above procedure is valid not only for braids on the disc, but also for braids on any compact, connected surface $S$ with boundary. Hence, this provides the first polynomial algorithm to solve the word problem in $P_n(S)$. The previously known algorithms~\cite{Scott, JuanJKTR}, based on usual braid combing, are clearly exponential.

In the case of closed surfaces, the braid combing is quite different. There is not such a decomposition of $P_n(S)$ as a semi-direct product of free groups, and one needs to use instead a decomposition $P_n(S)= \pi_1(S) \ltimes P_{n-1}(S\backslash \{p_1\})$, where $p_1$ is a point in $S$. The existence of such a decomposition was shown by Gon\c calves and Guaschi~\cite{GoncGuasch}, by giving a suitable section $s:\pi_1(S)\rightarrow P_n(S)$ of the natural projection $\pi: P_n(S)\rightarrow \pi_1(S)$ which they explained geometrically, but not algebraically. In this paper we provide an explicit algebraic section, for closed orientable surfaces of genus $g>0$, which does not coincide with the one defined in~\cite{GoncGuasch} (although we also give an explicit algebraic description of the section in~\cite{GoncGuasch}).

Finally, we explain why the procedure used to solve the word problem in surfaces with boundary does not generalize to closed surfaces in the natural way.

The plan of the paper is the following. In \cref{S:braids_on_surfaces} we introduce the basic notions of braids on surfaces. Then in \cref{sectioncombing} we explain braid combing in the case of surfaces with boundary, and give an example showing that combing is exponential. Straight line programs are treated in \cref{sectionstraightline}, and in \cref{sectioncompressed} we introduce the notion of {\it compressed braid combing} and give the polynomial algorithm to solve the word problem in braid groups on surfaces with boundary. \cref{sectioncombingclosed} deals with combing on a closed surface: We define the group section from $\pi_1(S)$ to $P_n(S)$ when $S$ is a closed surface, and we explain why the compressed braid combing cannot be generalized to this case in a natural way.

{\bf Acnkowledgements:} The first author thanks Saul Schleimer, for teaching him about straight line programs at the Centre de Recerca Matem\`atica (Barcelona) in 2012, and for useful conversations.


\section{Braids on surfaces}~\label{S:braids_on_surfaces}

Let $S$ be a compact, connected surface of genus $g$ and $p$ boundary components, and let $\mathcal{P} = \{p_1, \ldots, p_n\}$ be a set of $n$ distinct points of $S$. A \emph{geometric braid} on $S$ based at $\mathcal{P}$ is an n-tuple $\beta = (\gamma_1, \ldots, \gamma_n)$ of paths $\gamma_i : [0,1] \rightarrow S$, such that
\begin{itemize}
\item $\gamma_i(0) = p_i$, \,  $\forall i \in \{1, \ldots, n \}$,
\item $\gamma_i(1) \in \mathcal{P}$, \, $\forall i \in \{1, \ldots, n \}$,
\item $\{\gamma_1(t), \ldots, \gamma_n(t)\}$  are $n$ distinct points of $S$ \, $\forall t \in [0,1]$.
\end{itemize}

A braid on $S$ based at $\mathcal P$ is a homotopy class of such geometric braids (notice that homotopies must fix the endpoints).
The usual product of paths endows the set of braids with a group structure, and the resulting group (which is independent of the choice of $\mathcal{P}$) is called the \emph{braid group with $n$ strands on $S$}, and denoted $B_n(S)$. The path starting at $p_i$ will be called the $i$th strand of the braid.

The above definition can also be explained by saying that the braid group $B_n(S)$ is the fundamental group of $M_n(S)/\Sigma_n$, where
$$
    M_n(S)=\{(x_1,\ldots,x_n)\in S^n;\ x_i\neq x_j\; \forall i\neq j\}
$$
is the configuration space of $n$ distinct points in $S$, and $M_n(S)/\Sigma_n$ is the quotient of $M_n$ under the natural action of the symmetric group $\Sigma_n$ which permutes coordinates~\cite{Birman}. In other words, a braid can be seen as a motion of $n$ distinct points in $S$, whose initial configuration is $\mathcal P$, they move along the surface without colliding, and their final configuration is again $\mathcal P$ (though the particular position of each point in $\mathcal P$ may have changed). We will sometimes use this dynamic interpretation of a braid throughout this paper.

There are well known presentations for braid groups of surfaces.  In the particular cases of the sphere, the torus and the projective plane, the classical references are~\cite{FadellVanBuskirk, BirmanTorus, GoncGuaschiRP2}. For higher genus, one can find presentations in~\cite{Scott, JuanJKTR, PaoloPresentations, GoncGuaschNonOrient}. In all these presentations, some generators are related to the generators $\sigma_1,\ldots, \sigma_{n-1}$ of the classical braid group (and correspond to strand crossings), while some other generators are related to the generators of the fundamental group of $S$ (and correspond to motions of the distinguished points along the surface).

A braid $\beta$ is said to be pure if $\gamma_i(1) = p_i$, for all $i \in \{1, \ldots, n \}$, that is, if after the motion each distinguished point goes back to its original position. Pure braids form a finite index subgroup of $B_n(S)$, the \emph{pure braid group with n strands on $S$}, denoted $P_n(S)$. Notice that $P_n(S)$ is just the fundamental group of the configuration space $M_n(S)$. In particular, $B_1(S) = P_1(S) = \pi_1(S)$.

The main object of study in this paper is braid combing, which is a procedure to produce a particular normal form for pure braids. For this purpose, we need to make precise a particular presentation of $P_n(S)$. From now on we consider that $S$ is a compact, connected orientable surface. We will see that the combing in the non-orientable case is analogous, since one just need to consider the corresponding presentation of $P_n(S)$ appearing in \cite[Theorem 3]{GoncGuaschNonOrient}.

\begin{theorem} {\rm{\cite{PaoloPresentations}}} \label{teoprespure1} 
Let $S$ be an orientable surface of genus $g\geq 0$ with $p > 0$ boundary components. The group $P_n(S)$ admits the following presentation:
\begin{itemize}
\item Generators: \small{$\{A_{i,j} \ |  \ 2g+p \leq j \leq 2g+p+n-1, \ \ 1\leq i<j$\}.}
\item Relations:

\begin{itemize}
\item[(PR1)] $A^{-1}_{i,j}A_{r,s}A_{i,j} = A_{r,s}$ \hfill if $(i<j<r<s)$ \, or \, $(r+1<i<j<s)$ \\ \null \hfill   or \, $(i=r+1<j<s$ where $r\geq 2g$ or $r$ even{\rm )}.
\item[(PR2)] $A^{-1}_{i,j}A_{j,s}A_{i,j} = A_{i,s}A_{j,s}A^{-1}_{i,s}$ \hfill if $i<j<s$.
\item[(PR3)] $A^{-1}_{i,j}A_{i,s}A_{i,j} = A_{i,s}A_{j,s}A_{i,s}A^{-1}_{j,s}A^{-1}_{i,s}$ \hfill if $i<j<s$.
\item[(PR4)] $A^{-1}_{i,j}A_{r,s}A_{i,j} = A_{i,s}A_{j,s}A^{-1}_{i,s}A^{-1}_{j,s}A_{r,s}A_{j,s}A_{i,s}A^{-1}_{j,s}A^{-1}_{i,s}$ \\ \null \hfill if $(i+1 <~r~<~j<~s)$
 \\ \null \hfill or $(i+1 = r < j < s$ where $r>2g$ or $r$ odd{\rm )}.
\item[(ER1)] $A^{-1}_{r+1,j}A_{r,s}A_{r+1,j} = A_{r,s}A_{r+1,s}A^{-1}_{j,s}A^{-1}_{r+1,s}$ \hfill if $r$ odd, $r<2g$ and $j<s$.
\item[(ER2)] $A^{-1}_{r-1,j}A_{r,s}A_{r-1,j} = A_{r-1,s}A_{j,s}A^{-1}_{r-1,s}A_{r,s}A_{j,s}A_{r-1,s}A^{-1}_{j,s}A^{-1}_{r-1,s}$ \\ \null \hfill if $r$ even, $r\leq2g$ and $j<s$.
\end{itemize}
\end{itemize}
\end{theorem}

\medskip
\begin{remark}
We have corrected some missprints of the presentation appearing in~\cite{PaoloPresentations}.
\end{remark}

The generator $A_{i,j}$ can be represented as a motion of a single point of $\mathcal P$. Notice that $2g+p\leq j \leq 2g+p+n-1$, so we can write $j=(2g+p-1)+k$ for some $k=1,\ldots, n$. Then $A_{i,j}$ represents a motion of the point $p_k$ as shown in \cref{figgeneradoresPnS}. If $1\leq i\leq 2g$ the motion of $p_k$ corresponds to one of the classical generators of the fundamental group of a closed surface. If $i=2g+r$ with $r=1,\ldots,p-1$, the point $p_k$ moves around the $r$th boundary component (notice that there is no generator in which $p_k$ moves around the $p$th boundary component). If $i=(2g+p-1)+t$ for some $t=1,\ldots,k-1$, the point $p_k$ moves around the point $p_t$, as in the classical generators for the pure braid group of the disc~\cite{Birman}.

\begin{figure}[ht]
\centering
\includegraphics[width = 12.2cm]{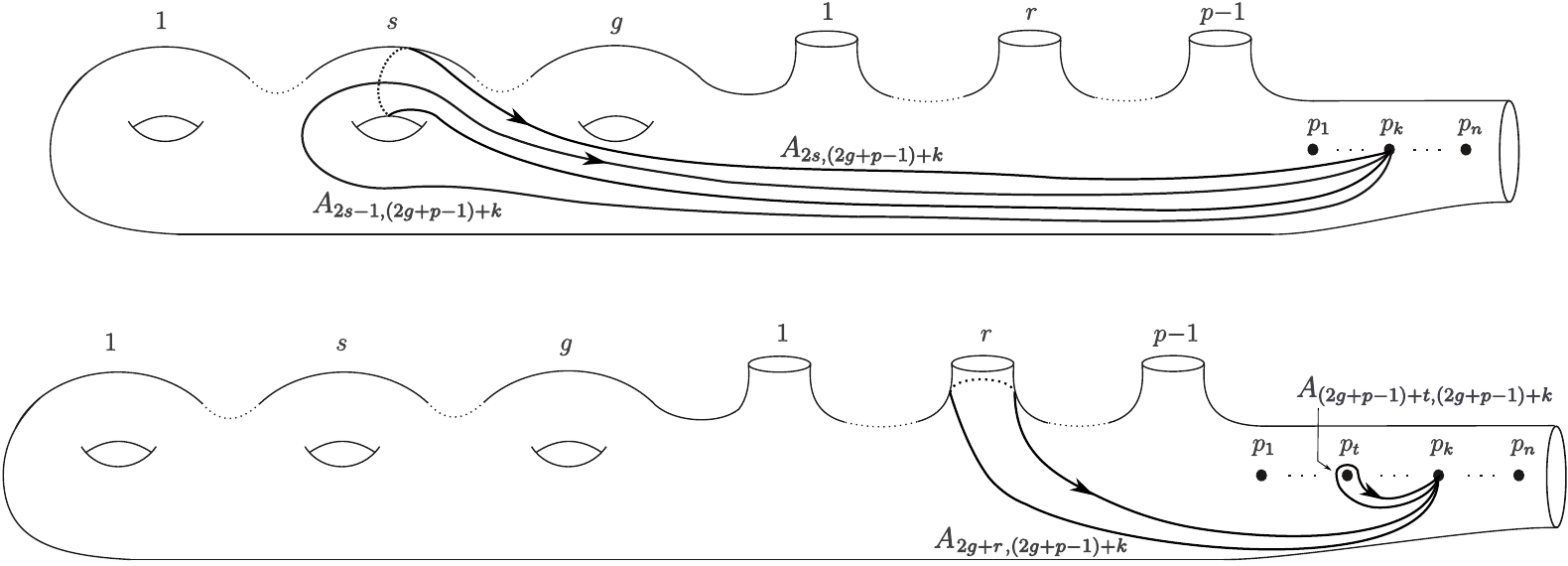}
\caption{\small{A geometric representation of each of the four different families of motions of $p_k$ represented by generators $A_{i,(2g+p-1)+k}$.}}
\label{figgeneradoresPnS}
\end{figure}

\begin{remark}The presentation given in~\cite{PaoloPresentations} is stated for $g>0$, but it also holds when $g=0$. For instance, if $g=0$ and $p=1$, $P_n(S)$ is the classical pure braid group $P_n$, the only relations that survive are $(PR1)-(PR4)$, and the presentation of \cref{teoprespure1} is precisely the presentation of $P_n$ given in~\cite{Birman}.
\end{remark}

\begin{remark}\label{remarkrelations}
By using the relations in \cref{teoprespure1} we can rewrite each conjugation $A_{i,j}^{-1}A_{r,s}^{\pm 1}A_{i,j}$ with $j<s$ as a word in generators whose second subindices equal $s$. Moreover, we can use these words to derive analogous relations allowing us to rewrite a word of the form $A_{i,j}A_{r,s}^{\pm 1}A_{i,j}^{-1}$ with $j<s$ as a word in generators whose second subindices equal $s$:
\begin{itemize}
\item[ ] \begin{itemize}
\item[(PR1$'$)] $A_{i,j}A_{r,s}A_{i,j}^{-1} = A_{r,s}$ \hfill if $(i<j<r<s)$ \, or \, $(r+1<i<j<s)$ \\ \null \hfill   or \, $(i=r+1<j<s$ where $r\geq 2g$ or $r$ even{\rm )}.
\item[(PR2$'$)] $A_{i,j}A_{j,s}A_{i,j}^{-1} = A_{j,s}^{-1}A_{i,s}^{-1}A_{j,s}A_{i,s}A_{j,s}$ \hfill if $i<j<s$.
\item[(PR3$'$)] $A_{i,j}A_{i,s}A_{i,j}^{-1} = A_{j,s}^{-1}A_{i,s}A_{j,s}$ \hfill if $i<j<s$.
\item[(PR4$'$)] $A_{i,j}A_{r,s}A_{i,j}^{-1} = A^{-1}_{j,s}A^{-1}_{i,s}A_{j,s}A_{i,s}A_{r,s}A^{-1}_{i,s}A^{-1}_{j,s}A_{i,s}A_{j,s}$ \\ \null \hfill if $(i+1 <~r~<~j<~s)$
 \\ \null \hfill or $(i+1 = r < j < s$ where $r>2g$ or $r$ odd{\rm )}.
\item[(ER1$'$)] $A_{r+1,j}A_{r,s}A_{r+1,j}^{-1} = A_{r,s}A_{j,s}$ \hfill if $r$ odd, $r<2g$ and $j<s$.
\item[(ER2$'$)] $A_{r-1,j}A_{r,s}A_{r-1,j}^{-1} = A^{-1}_{j,s}A_{r,s}A_{r-1,s}^{-1}A_{j,s}^{-1}A_{r-1,s}A_{j,s}$ \\ \null \hfill if $r$ even, $r \leq 2g$ and $j<s$.
\end{itemize}
\end{itemize}

These relations together with those appearing in \cref{teoprespure1} are used frequently throughtout this paper. We use the expression \textit{(PR/ER)-relations} to denote the set consisting of these 12 types of relations.
\end{remark}

When $S$ is a closed surface, one needs to add an extra relation in the presentation of $P_n(S)$.  We write $[a,b]=aba^{-1}b^{-1}$.

\begin{theorem} {\rm{\cite{PaoloPresentations}}} \label{teoprespure2}
Let $S$ be an orientable closed surface of genus $g\geq 0$. The group $P_n(S)$ admits a presentation with generators $$\{A_{i,j} \ | \  1 \leq i \leq 2g + n - 1, \ \ 2g+1 \leq j \leq 2g+n, \ \ i<j\}$$
and the same relations as those in \cref{teoprespure1} together with\\
\hspace{1cm} (TR) {\small $[A^{-1}_{2g,2g+k}, A_{2g-1,2g+k}] \cdots [A^{-1}_{2,2g+k}, A_{1,2g+k}] \, = \, \displaystyle \prod_{l=2g+1}^{2g+k-1} A_{l,2g+k} \prod_{j=2g+k+1}^{2g+n} A_{2g+k,j},$}\\
for $k=1, \ldots, n$.
\end{theorem}

Notice that this time the generator $A_{i,j}$ corresponds to a motion of the point $p_k$ if $j=2g+k$. The picture corresponds exactly to the case in which $p=1$ in \cref{figgeneradoresPnS}, provided one removes the boundary component on the right hand side of the picture. Note that the (PR/ER)-relations also hold in this case.

As above, \cref{teoprespure2} was stated in~\cite{PaoloPresentations} for genus $g>0$, but the presentation also holds in the case $g=0$, as one gets the known presentation for the pure braid group of the sphere given in~\cite{GilleteVanBuskirk}.

\begin{remark} Choose some points $\mathcal Q=\{q_1,\ldots, q_m\}$ in a compact, connected surface $S$, and consider the non-compact surface $S\backslash \mathcal Q$. Then the braid groups $P_n(S\backslash \mathcal Q)$ and $B_n(S\backslash \mathcal Q)$ are naturally isomorphic to $P_n(S')$ and $B_n(S')$, respectively, where $S'$ is the surface obtained from $S$ by removing a small open neighborhood of $\mathcal Q$, that is, by replacing each $q_i$ with a boundary component. In other words, removing a point from $S$ is equivalent to adding a boundary component, as far as braid groups on the surface are concerned.
\end{remark}

\section{Braid combing on a surface with boundary}\label{sectioncombing}

From now on we assume that $S$ is an orientable surface with $p>0$ boundary components, unless otherwise stated. The case when $S$ is a non-orientable surface with boundary can be treated analogously, since one just needs to consider the relations apperaring in the presentation of $P_n(S)$ \cite[Theorem 3]{GoncGuaschNonOrient}, instead of the $(PR/ER)$-relations. We will discuss the case when $S$ is an orientable closed surface in \cref{sectioncombingclosed}.

Braid combing is a process by which a particular normal form of a pure braid is obtained. It was introduced by Artin~\cite{Artin} for the pure braid groups of the disc, but it can be generalized to pure braid groups of other surfaces, as we will explain in this section.

Recall that $\mathcal P=\{p_1,\ldots,p_n\}$ are $n$ distinguished points in the surface $S$. For $m=1,\ldots,n$, we denote $\mathcal P_m=\{p_1,\ldots,p_m\}$ and $\mathcal Q_{n-m}=\{p_{m+1},\ldots,p_n\}$.  It is well known (see for instance~\cite{GuaschiJuanPineda}) that the map from $P_n(S)$ to $P_m(S)$ which `forgets' the last $n-m$ strands determines a short exact sequence
\begin{equation}~\label{ESnm}
   1 \rightarrow P_{n-m}(S\backslash \mathcal P_m) \stackrel{i_{n,m}}{\longrightarrow} P_n(S) \stackrel{p_{n,m}}{\longrightarrow} P_m(S) \rightarrow 1
\end{equation}
where the base points of the three involved groups are, respectively, $\mathcal Q_{n-m}$, $\mathcal P_n$ and $\mathcal P_{m}$.

In the case of a closed surface, the above sequence is also exact if we assume that $m\geq 3$ if $S$ is the sphere $\mathbb S^2$, and that $m\geq 2$ if $S$ is the projective plane $\mathbb RP^2$. The bad cases are those which involve finite groups, as $P_1(\mathbb S^2)=P_2(\mathbb S^2)=1$, $P_3(\mathbb S^2)=\mathbb Z/2\mathbb Z$~\cite{FadellVanBuskirk}, and also $P_1(\mathbb RP^2)=\mathbb Z/2\mathbb Z$ and $P_2(\mathbb RP^2)=\mathcal Q_8$, the quaternion group of order 8~\cite{GuaschiJuanPineda}.

In the sequence~\labelcref{ESnm}, an element of $P_{n-m}(S\backslash \mathcal P_m)$ can be seen as a braid in $P_n(S)$ in which the first $m$ strands are trivial or, in other words, in which the points of $\mathcal P_{m}$ do not move. This is equivalent to consider that $S$ has $m$ extra punctures (or, removing a small neighborhood around each fixed puncture, that $S$ has $m$ extra boundary components), and $n-m$ points are moving.

To define braid combing, we single out a particular case of the exact sequence~\labelcref{ESnm}. When $m=n-1\geq 1$ one has
\begin{equation}\label{sucespure}
1 \, \rightarrow  \pi_1 (S \setminus \mathcal{P}_{n-1})   \stackrel{i_{n,n-1}}{\longrightarrow}  P_n(S) \stackrel{p_{n,n-1}}{\longrightarrow}  P_{n-1}(S)  \rightarrow  1,
\end{equation}
where we applied that  $P_1 (S \setminus \mathcal{P}_{n-1})=\pi_1 (S \setminus \mathcal{P}_{n-1})$. We can easily describe the injection $i_{n,n-1}$ algebraically: Let $j_n=(2g+p-1)+n$. Considering each element of $\mathcal P_{n-1}$ as a boundary component and using the generators of \cref{teoprespure1}, we see that $i_{n,n-1}(A_{i,j_n})=A_{i,j_n}$ for all $i=1,\ldots, j_n-1$. The projection $p_{n,n-1}$ is also very easy to describe: For every $i<j$, the element $p_{n,n-1}(A_{i,j})$ is either 1 (if $j=j_n$), or $A_{i,j}$ (if $j<j_n$).

It is known that, since $S$ has nontrivial boundary, the sequence~\labelcref{sucespure} splits~\cite{GoncGuasch}. An explicit section $s:\: P_{n-1}(S)\rightarrow P_n(S)$ is given by $s(A_{i,j})=A_{i,j}$ for all $i,j$.

Now notice that, as $n-1\geq 1$, the fundamental group of $S \setminus \mathcal{P}_{n-1}$ is a free group of rank $2g+p+n-2$~\cite{Hatcher}, so we have an exact sequence:
\begin{equation}\label{ESn,n-1}
1 \, \rightarrow  \mathbb{F}_{2g+p+n-2}   \stackrel{i_{n,n-1}}{\longrightarrow}  P_n(S) \stackrel{p_{n,n-1}}{\longrightarrow}  P_{n-1}(S)  \rightarrow  1.
\end{equation}
Therefore, $P_n(S) = P_{n-1}(S) \ltimes \mathbb{F}_{2g+p+n-2}$, where $P_{n-1}(S)$ can be seen as the subgroup of $P_n(S)$ generated by $\{A_{i,j}\}_{i<j< j_n}$, and $\mathbb{F}_{2g+p+n-2}$ as the subgroup of $P_n(S)$ generated by $\{A_{i,j_n}\}_{i<j_n}$.

By induction on $n$,  $P_n(S)$ can be written as an iterated semi-direct product of free groups:
$$
P_n(S)= ((\cdots ((\mathbb{F}_{2g+p-1} \ltimes \mathbb{F}_{2g+p}) \ltimes \mathbb F_{2g+p+1} )\ltimes \cdots \mathbb F_{2g+p+n-3}) \ltimes  \mathbb{F}_{2g+p+n-2}.
$$
The process of computing the normal form of a braid ({\bf on a connected, compact surface with boundary}) with respect to this iterated semi-direct product, is known as {\it braid combing}.

\begin{definition}\label{D:cnf_with_boundary}
Let $S$ be an orientable surface of genus $g$ with $p>0$ boundary components. The {\bf combed normal form} of a braid $\alpha \in P_n(S)$ is a decomposition
$$
   \alpha=\alpha_1\alpha_2\cdots \alpha_n
$$
where, for $k=1,\ldots,n$, $\alpha_k$ belongs to the subgroup generated by $\{A_{i,j_k}\}_{i<j_k}$, with $j_k=(2g+p-1)+k$.
\end{definition}

By the uniqueness of normal forms with respect to semi-direct products, the combed normal form of a braid is unique. Also, since each of the subgroups described in \cref{D:cnf_with_boundary} is a free group on the given generators, one can choose a unique reduced word $w_i$  to represent each $\alpha_i$, and this gives a unique word representing $\alpha$.

We now describe explicitly how to compute the combed normal form of a pure braid $\alpha$ represented by a given word $w$ in the generators of $P_n(S)$. For that purpose, we just need to find a way to {\it move} the letters $A^{\pm 1}_{i,j}$ with smaller second index to the left of those with bigger second index. This can be achieved by using the (PR/ER)-relations to replace two consecutive letters $AB$ where $A=A^{\pm 1}_{r,s}$,  $B=A^{\pm 1}_{i,j}$ and $j<s$, with the word $B W$, where $W$ is a word formed by letters whose second subindex equals $s$. Iteratively applying these substitutions (together with free reduction), the word $w$ can be transformed into a reduced word of the form $w_1w_2\cdots w_n$, which also represents $\alpha$, such that the second subindex of every letter of $w_k$ is $j_k$, for $k=1,\ldots,n$. The decomposition $\alpha=\alpha_1\cdots \alpha_n$, where $\alpha_k$ is the braid represented by $w_k$ for $k=1,\ldots,n$, is thus the combed normal form of $\alpha$.

We will see in \cref{subsectionexponencial} that the length of the word $w_1\cdots w_n$ is possibly exponential with respect to the length of $w$. But we will show that it is possible to represent $w_1\cdots w_n$ in a {\it compressed} way. The key point consists of moving the letters as explained in the above paragraph but, instead of applying the conjugations described in (PR/ER)-relations, we will keep track of those conjugations without applying them. We will now show how to do this. To avoid cumbersome notation we use the expression $u^v$ to denote the word $v^{-1} u v$, with $u$ and $v$ two given words.

Suppose that $w=u_1u_2\cdots u_m$, where $u_t=A^{\pm 1}_{i_t,s_t}$ for $t=1,\ldots,m$. Thus $w$ is a sequence of $m$ letters. For every $t$, let $v_t$ be the subsequence of $w$ formed by those letters $u_r$ such that $r>t$ and $s_r<s_t$. In other words, $v_t$ is formed by the letters of $w$ which come after $u_t$ and have smaller second index. Notice that $v_t$ is formed precisely by the letters of $w$ that must be swapped with $u_t$ when combing the braid.

By construction, we have the following:

\begin{lemma}
Let $w=u_1u_2\cdots u_m$ be as above, representing a braid $\alpha$. Given $k\in \{1,\ldots,n\}$ let $u_{i_1}u_{i_2}\cdots u_{i_t}$ be the subsequence of $w$ formed by those letters whose second subindex is $j_k$. Define:
$$
    \overline w_k = u_{i_1}^{v_{i_1}}u_{i_2}^{v_{i_2}}\cdots u_{i_t}^{v_{i_t}}.
$$
Then $\overline w_k$ represents the same braid as $w_k$, so $\overline w=\overline w_1\cdots \overline w_n$ represents the combed normal form of $\alpha$.
\end{lemma}

\begin{example}\label{excodificar}
If $w = A_{1,2}A_{1,4}A_{1,2}A_{2,3}^{-1}A_{2,4}A_{1,3}A_{1,2} \in P_4$, we have
$$
\overline{w} = \underbrace{A_{1,2}A_{1,2}A_{1,2}}_{\overline{w}_2} \underbrace{{A_{2,3}^{-1}}^{A_{1,2}}A_{1,3}^{A_{1,2}}}_{\overline{w}_3} \underbrace{A_{1,4}^{A_{1,2}A_{2,3}^{-1}A_{1,3}A_{1,2}}A_{2,4}^{A_{1,3}A_{1,2}}}_{\overline{w}_4}.
$$
Note that in $P_4$ we have $g=0$ and $p=1$, so $j_k=k$. There are no letters of the form $A_{i,1}$, that is why $\overline w_1$ is the trivial word.
\end{example}

The word $\overline w=\overline w_1\cdots \overline w_n$ is not too long with respect to $w$: its length is at most $m^2$. Indeed, the worst case occurs if we need to swap every pair of letters of $w$, so each $v_i$ has length $m-i$, and thus the length of $\overline w$ is $m+2\frac{m(m-1)}{2}= m^2$.

Let us see how we can obtain enough information to describe each $\overline w_k$, by a procedure which is linear in $m$.

\begin{lemma}\label{L:short_w_k_bar_with_boundary}
Given $w=u_1\cdots u_m$ as above, and given $k\in \{1,\ldots, n\}$, one can compute in time $O(m)$ a subsequence $v$ of $w$ and a list of pairs of integers $(c_1,d_1),\ldots,(c_t,d_t)$, where $1\leq |c_r| \leq m$ and $0\leq d_r\leq m$ for each $r=1,\ldots,t$ $(t<m)$, which encodes all the information to describe $\overline w_k$.
\end{lemma}

\begin{proof}
Going through the word $w=u_1\cdots u_m$ once, we can determine the position $i_1$ of the first letter with second index $j_k$, and the word $v=v_{i_1}$. We know that $\overline w_k$ will have the form $u_{i_1}^{v_{i_1}}u_{i_2}^{v_{i_2}}\cdots u_{i_t}^{v_{i_t}}$ and, by construction, that each $v_{i_r}$ is a suffix of $v_{i_1}$. Hence, in order to determine $v_{i_r}$ we just need to provide the word $v_{i_1}$ and the length $d_r$ of $v_{i_r}$. Also, in order to determine the letter $u_{i_r}$ we just need to provide its first index and the sign of its exponent ($\pm 1$). Hence we just need to provide an integer $c_r$ such that $|c_r|$ is the first index of $u_{i_r}$, and whose sign is equal to the exponent of $u_{i_r}$.

To obtain these integers, we go through $w$ again, starting at the position $i_1$ and setting $d$ equal to the length of $v_{i_1}$. Every time we read a letter $u$ whose second index is smaller than $j_k$, we decrease $d$ by one. Every time we read a letter $u$ whose second index is $j_k$, we store the pair $(c,d)$ where $|c|$ is the first index of $u$ and the sign of $c$ is given by the exponent of $u$ (positive if $u=A_{|c|,j_k}$, negative if $u=A^{-1}_{|c|,j_k}$), and $d$ is the above number. Notice that the first step stores $(c_1,d_1)$ where $|c_1|=i_1$ and $d_1$ is the length of $v_{i_1}$.

At the end of the whole procedure, we have gone through $w$ twice (so the complexity is $O(m)$), and we have obtained a word $v_{i_1}$ and a list of pairs of integers $(c_1,d_1),\ldots, (c_t,d_t)$ which encode $\overline w_k$ as desired.
\end{proof}

\begin{example}
Let $w$ be the word with 7 letters appearing in \cref{excodificar}. To describe $\overline{w}_3$ we store the word $A_{1,2}$ and the pairs $(-2,1), (1,1)$. To describe $\overline{w}_4$ we store the word $A_{1,2}A_{2,3}^{-1}A_{1,3}A_{1,2}$ and the pairs $(1,4), (2,2)$.
\end{example}

We have then obtained, for every $k=1,\ldots,n$, a {\it short} word $\overline w_k$ representing the factor $\alpha_k$ of the braid combing. But a word representing $\alpha_k$ will be useful only if we represent it as a reduced word $w_k$ in the generators $\{A_{i,j_k}\}_{i<j_k}$ of the corresponding free group. We can obtain $w_k$ from $\overline w_k$  by iteratively applying the (PR/ER)-relations: Given a word $u_i^{v_i}$ (suppose that the second subindex of $u_i$ is $j_k$), we can write $v_i=a_1\cdots a_t$; then we conjugate $u_i$ by $a_1$ using the (PR/ER)-relations, obtaining an element which can be written as a word whose letters have second index equal $j_k$. Next we conjugate this word by $a_2$, using the relations again (taking into account that conjugating a product by $a_2$ is the same as conjugating each factor by $a_2$). And so on. At the end, $u_i^{v_i}$ is transformed into a word in which the second index of every letter is $j_k$. After repeating this process with every word in $\overline w_k$ and applying free reduction, one obtains $w_k$.

This last step of braid combing is the exponential one! It can produce a word $w_k$ which is exponentially long with respect to $m$ (we will show such an example in \cref{Lema_exp}). For this reason, in \cref{sectioncompressed} we introduce a method based on straight line programs to avoid this last step of braid combing, and to solve the word problem for braids on surfaces with boundary in polynomial time.

\subsection{Braid combing is exponential}\label{subsectionexponencial}

In this section we provide an example to show that braid combing is, in general, an exponential procedure. That is, we present a family of pure braids $\beta_m$, $m\geq 1$, where each $\beta_m$ can be expressed as a word whose length is linear in $m$, but whose combed expression is a word of exponential length with respect to $m$. This implies that there is no hope to produce an efficient algorithm to comb braids, if one wishes to express combed braids as words in usual braid generators.

The example will be given in $P_4$, the pure braid group with 4 strands on the disc $\mathbb D^2$, with generators $A_{i,j}=\sigma_{j-1}\sigma_{j-2}\cdots \sigma_{i+1} \sigma_i^2 \sigma_{i+1}^{-1}\cdots \sigma_{j-1}^{-1}$, for $1\leq i < j \leq 4$. As $P_4$ embeds in $P_n(S)$ for $n\geq 5$ (where $S$ is an orientable surface, or $n\geq 4$ if $S$ has genus $g>0$) by an embedding which sends generators to generators (of the presentations defined in \cref{S:braids_on_surfaces}), it follows that, in general, braid combing is exponential in surface braid groups.

Given $m\geq 1$, define:
$$
     \beta_m=\left(A_{1,2}^{-1}A_{2,3}\right)^{-m}A_{3,4}\left(A_{1,2}^{-1}A_{2,3}\right)^{m}\in P_4.
$$
The length of the given word representing $\beta_m$ is $4m+1$, so it is linear in $m$. We can then show the following:

\begin{lemma}\label{Lema_exp}
With the above notation, the combed normal form of $\beta_m$ is a reduced word in $A_{1,4},A_{2,4},A_{3,4}$ and their inverses, whose length is exponential in $m$.
\end{lemma}

\begin{proof}
We consider the inner automorphism $\phi$ of $P_4$ which consists of conjugating by $A_{1,2}^{-1}A_{2,3}$. Then $\beta_m=\phi^m(A_{3,4})$. Using the short exact sequences explained in the previous section, we see that $\phi$ induces an automorphism on the free group generated by $A_{1,4}$, $A_{2,4}$, $A_{3,4}$, so $\beta_m$ can be written as a word on these letters and their inverses.

The fact that $\beta_m$ is exponentially long when written as a reduced word in $A_{1,4},A_{2,4},A_{3,4}$ and their inverses, can be deduced from the fact that $A_{1,2}^{-1}A_{2,3}$ is a pseudo-Anosov braid in $P_3$. However, since we want to make this example as explicit as possible, we will provide an algebraic proof.

To simplify the notation, let us denote $a_i=A_{i,4}$ for $i=1,2,3$. From the presentation of $P_4$, one can check the following:
$$
\begin{array}{l}
  \phi(a_1)=a_2a_3a_2^{-1}a_3^{-1}a_2^{-1}a_1a_2a_3a_2a_3^{-1}a_2^{-1}
\\
  \phi(a_2)=a_2a_3a_2^{-1}a_3^{-1}a_2^{-1}a_1^{-1}a_2a_3a_2a_3^{-1}a_2^{-1}a_1a_2a_3a_2a_3^{-1}a_2^{-1}
\\
  \phi(a_3)=a_2a_3a_2^{-1}
\end{array}
$$
Let us consider the following three words: $x=a_1$, $y=a_2a_3a_2^{-1}$ and $w=a_1a_2a_3$. Then we can write:
$$
\begin{array}{l}
  \phi(a_1)=yw^{-1}xwy^{-1}
\\
  \phi(a_2)=yw^{-1}x^{-1}wy^{-1}wy^{-1}
\\
  \phi(a_3)=y
\end{array}
$$

 Suppose that we can decompose a braid in the following way:
$$
  \alpha= z_1w^{e_1}z_2w^{e_2}\cdots z_{t-1}w^{e_{t-1}}z_t,
$$
where $z_i\in\{x,y,x^{-1},y^{-1}\}$ and $e_i\in \{1,-1\}$ for $i=1,\ldots,t-1$. This expression of $\alpha$ is not necessarily reduced when written in terms of $\{a_i^{\pm 1}\}$, but the cancellations are only restricted to subwords of the form $x^{-1}w$ or $w^{-1}x$, and in these cases only two letters cancel, and no further cancellation is produced. Therefore, the length of the reduced word $\overline{\alpha}$ in $\{a_i^{\pm 1}\}$ associated to $\alpha$, is at least $2t-1$. If we moreover assume that $z_1=y$, then the length of $\overline{\alpha}$ is greater than $2t$. In this case, we say that $\alpha$ admits a $xyz$-decomposition of length $t$.

We point out that $\phi(w)=w$. This follows from the fact that $w=a_1a_2a_3= \sigma_3\sigma_2\sigma_1\sigma_1\sigma_2\sigma_3$ commutes with $A_{1,2}=\sigma_1^2$ and $A_{2,3}=\sigma_2^2$.

Finally, we conclude the proof by showing that $\beta_m$ admits a $xyz$-decomposition of length $t\geq 3^{m-1}$.


We proceed by induction on $m$. First, we check that $\beta_1=\phi(a_3)=y$, so the claim holds in the case $m=1$. Now suppose the claim is true for some $m\geq 1$, so $\beta_m=z_1w^{e_1}z_2w^{e_2}\cdots z_{t-1}w^{e_{t-1}}z_t$ for some $t\geq 3^{m-1}$. Then we have
$$
\beta_{m+1}=\phi(\beta_m) = \phi(z_1)w^{e_1}\phi(z_2)w^{e_2}\cdots \phi(z_{t-1})w^{e_{t-1}}\phi(z_t).
$$
Finally, since
$$
  \phi(x)=yw^{-1}xwy^{-1},\quad\mbox{and}\quad \phi(y)=yw^{-1}x^{-1}wy^{-1}w yw^{-1}yw^{-1}xwy^{-1},
$$
it follows that $\beta_{m+1}$ admits a $xyz$-decomposition, whose length is at least $3t$, thus greater than $3^m$.
\end{proof}

\section{Straight line programs}\label{sectionstraightline}

We have seen that combing a braid $\alpha$ consists of decomposing it as a product $\alpha_1\alpha_2\cdots \alpha_n$ in a suitable way. We have also seen that, in general, if $\alpha$ is given as a word of length $m$, the length of a factor $\alpha_i$ may be exponential in $m$. So, how could we make this procedure to have polynomial complexity? The answer is: We will not describe each $\alpha_i$ as a word in the generators of the corresponding free group. Instead, we will describe each $\alpha_i$ as a {\it compressed word} (also called {\it straight line program}).

The concept of straight line program is well known in Complexity Theory~\cite{BCS}, and has been used in Combinatorial Group Theory to reduce the complexity of decision problems~\cite{Schleimer}.  In this section we review the main aspects related to straight line programs, which we will call {\it compressed words}, following~\cite{Schleimer}.

Roughly speaking, a \textit{compressed word} $\mathbb{A}$ consists on two disjoint finite sets of symbols $\mathcal{F}$ and $\mathcal{A}$ (the latter is ordered), called terminal and non-terminal alphabets respectively, together with a set of production rules indicating how to rewrite each non-terminal character in $\mathcal{A}$ as a word in $\mathcal{F}$ and smaller characters of $\mathcal A$. In this way, the biggest non-terminal character can be rewritten, using the production rules, as a word in $\mathcal{F}$ that we denote $ev(\mathbb{A})$ (the {\it evaluation} of $\mathbb A$, or the {\it decompressed word}). So $\mathbb A$ can be seen as a small set of {\it instructions} on how to produce a long word $ev(\mathbb{A})$ in $\mathcal{F}$.

Here is the rigorous definition:

\begin{definition}
A compressed word (or straight line program) $\mathbb{A}$ consists on a finite alphabet $\mathcal{F}$ of terminal characters together with an ordered finite set of (non-terminal) symbols $\mathcal{A} = \{A_1, \ldots, A_n\}$, and a set of production rules $\mathcal{P} = \{A_i \rightarrow W_i \}_{1 \leq i \leq n}$ allowing to replace each non-terminal $A_i \in \mathcal{A}$ with its production: a (possibly empty) word $W_i \in (\mathcal{F} \cup \mathcal{A})^*$, where every non-terminal $A_j$ appearing in $W_i$ has index $j<i$. The greatest non-terminal character in $\mathbb{A}$, $A_n \in \mathcal{A}$, is called the root.
\end{definition}

The {\it evaluation} of the compressed word $\mathbb{A}$, $ev(\mathbb{A})$, is the (decompressed) word in $\mathcal{F}$ obtained by replacing successively every non-terminal symbol with the right-hand side of its production rule, starting from the root.

We define the {\it size} of a compressed word $\mathbb A$ as $|\mathbb A|=\sum_i |W_i|$. We can assume that $|\mathbb A|\geq \#(\mathcal F)$ (if a terminal character appears in no production rule, we can remove it from $\mathcal F$ as it will not appear in $ev(\mathbb A)$). We can also assume that $|\mathbb A|\geq \#(\mathcal A)$ (if a production rule transforms a non-terminal symbol into the empty word, we can remove every appearance of the non-terminal symbol from the whole compressed word). Therefore, the space needed to store $\mathbb A$ is at most $2 |\mathbb A|$, and this is the reason why $|\mathbb A|$ is called ``the size of $\mathbb A$''.

\begin{example}\label{excompressed}
Given $n\geq 1$, consider the following compressed word:
$$
\mathbb{A}_n:= \,  \left\<
\begin{array}{l}
\mathcal{F} = \{a\} \\
\mathcal{A} = \{A_1, \ldots, A_n \} \\
\mathcal{P} = \{A_i \rightarrow A_{i-1} A_{i-2}\}_{i = 3}^{n} \cup \{A_2 \rightarrow a\} \cup \{A_1 \rightarrow b\}
\end{array}
\right\>.
$$

The sequence of compressed words $\{\mathbb A_n\}_{n\geq 1}$ encodes the sequence of {\em Fibonacci words}. The first seven decompressed words $ev(\mathbb A_1),\ldots, ev(\mathbb A_7)$ are, respectively:
$$
b, \quad a, \quad ab, \quad aba, \quad abaab, \quad abaababa, \quad abaababaabaab.
$$
We see that every word is the concatenation of the two previous ones. Notice that $|\mathbb A_n|=2n-2$, while $|ev(\mathbb A_n)|=F_n$, the $n$-th Fibonacci number. So, in this example, the sizes of compressed words grow linearly in $n$, while the sizes of decompressed words grow exponentially.
\end{example}

Two compressed words $\mathbb A$ and $\mathbb B$ are said to be {\it equivalent} if $ev(\mathbb{A})=ev(\mathbb{B})$.

A crucial property is that a pair of compressed words can be {\it compared} without being decompressed:

\begin{theorem}\label{T:Plandowski}{\rm{\cite{Plandowski}}}
Given two compressed words $\mathbb{A}$ and $\mathbb{B}$, there is a polynomial time algorithm  in $|\mathbb{A}|$  and $|\mathbb{B}|$ which decides whether or not $ev(\mathbb{A}) = ev(\mathbb{B})$.
\end{theorem}

Now recall that we want to use compressed words to describe the factors $\alpha_1\cdots \alpha_n$ of a combed braid. The leftmost factor $\alpha_1$ does not need to be compressed, as it can be expressed as a subsequence of the original word. Each of the other factors, $\alpha_2,\ldots, \alpha_n$, belongs to a free group. So we would like to compress words representing elements of a free group.

In order for this procedure to be useful (for instance, to solve the word problem in braid groups), we need to be able to compare two {\it compressed} combed braids. This means that we need to determine whether two compressed words ($\mathbb A$ and $\mathbb B$) evaluate to words ($ev(\mathbb A)$ and $ev(\mathbb B)$) which represent the same element in a free group.  That is, we want to compare the uncompressed words not just as words, but as elements in a free group.

Fortunately, this problem has already been satisfactory solved:

\begin{theorem}\label{T:reduce_SLP}{\rm{\cite{Lohrey}}}
Given a compressed word $\mathbb{A}$, there exists a polynomial time algorithm producing a compressed word $\mathbb A_{red}$ such that $ev(\mathbb{A}_{red})$ is the free reduction of $ev(\mathbb{A})$.
\end{theorem}

As a consequence, if we have two compressed words $\mathbb A$ and $\mathbb B$, we can compute $\mathbb A_{red}$ and $\mathbb B_{red}$ in polynomial time, and then we can check in polynomial time whether $ev(\mathbb{A}_{red})$ and $ev(\mathbb{B}_{red})$ are the same word (\cref{T:Plandowski}). Hence, the {\it compressed word problem} in a free group is solvable in polynomial time.

We will therefore be able to use compressed words to perform braid combing, and to compare combed braids, in polynomial time. We now proceed to describe how to compress the words appearing in the process of braid combing.


\section{Compressed braid combing}\label{sectioncompressed}

Throughout this section, $S$ will be a compact, connected orientable surface with $p>0$ boundary components. The arguments in this section can also be applied in a straightforward way if $S$ is non-orientable, just by using the appropriate presentations.

Recall that the combing algorithm for a braid $\alpha\in P_n(S)$ starts with a word $w=u_1\cdots u_m$ representing $\alpha$, where $u_k=A^{\pm 1}_{i,j}$ for $k=1,\ldots, m$, and produces $n$ words, $w_1,\ldots, w_n$, representing the factors $\alpha_1,\ldots, \alpha_n$ of the combed normal form of $\alpha$. Each $w_k$ is a reduced word formed by letters whose second subindex is $j_k$. Moreover, we gave a procedure to compute short words $\overline w_1, \ldots, \overline w_n$ representing the factors $\alpha_1,\ldots, \alpha_n$ of the combed normal form of $\alpha$ in polynomial time, only that the second index of a letter in $\overline w_k$ is not necessarily $j_k$.

In some sense $\overline w_k$ is a {\it compressed} expression of $w_k$. Notice that the word $\overline w_1$ is actually equal to $w_1$, and its length is at most $m$, so we do not need to use compression for this first factor. Let us study the other cases.

In this section we will see how, starting from $\overline w_k$, one can define a compressed word $\mathbb A_k$ (a straight line program) such that $w_k$ is the free reduction of $ev(\mathbb A_k)$. Moreover, the size of $\mathbb A_k$ will be of order $O((g+p+n)m)$. This will allow us to determine and compare the factors of the combed normal forms of braids in $P_n(S)$, without needing to evaluate them, so it solves the word problem in $P_n(S)$ in polynomial time.

Let $k\in \{2,\ldots,n\}$, and recall that $\overline{w}_k = u_{i_1}^{v_{i_1}} \cdots u_{i_t}^{v_{i_t}}$, where each $u_{i_r}$ belongs to $\{A^{\pm 1}_{i,j_k}\}_{i<j_k}$, and each $v_{i_r}$ is a word formed by letters from $\{A^{\pm 1}_{i,j_s}\}_{i<j_s<j_k}$. Recall also that, by construction, each $v_{i_r}$ is a suffix of $v_{i_1}$. Also, $t\leq m$ and the length of $v_{i_1}$ is smaller than $m$ (where $m$ is the length of $w$).

By \cref{L:short_w_k_bar_with_boundary}, we can obtain in time $O(m)$ a subsequence $v_{i_1}$ of $w$ and a sequence of pairs of integers $(c_1,d_1), \ldots, (c_t,d_t)$ which encode $\overline w_k$.

We want to define a compressed word $\mathbb{A}_k$ representing $\alpha_k$, so the terminal alphabet will consist of the generators of $P_n(S)$ corresponding to the motion of the point $p_k$, that is, $\{A_{i,j_k}^{\pm 1}\}_{i<j_k}$. To be consistent with the forthcoming notation, we will denote $X_{i,0}=A_{i,j_k}$ and $X_{-i,0}=A^{-1}_{i,j_k}$, for $i=1,\ldots,j_k-1$. So the terminal alphabet of $\mathbb A_k$ becomes $\mathcal F_k=\{X_{i,0}\}_{0<|i|<j_k}$.

On the other hand, the non-terminal symbols of $\mathbb A_k$ will consist of a single symbol $X_k$ (the root), plus a symbol corresponding to $u^v$ for each $u\in \mathcal F_k$ and each nontrivial suffix $v$ of $v_{i_1}$. We know that the word $u^v$ can be determined by a pair of integers $(c,d)$, where $u=X_{c,0}$ and $v$ is the suffix of $v_{i_1}$ of length $d$. Hence, the set of nonterminal symbols is
$$
  \mathcal A_k=\{X_{c,d}\}_{\genfrac{}{}{0pt}{}{0<|c|<j_k}{0 < d \leq |v_{i_1}|}}\cup \{X_k\}.
$$
We order the elements of $\mathcal A_k$ (distinct form $X_k$) according to their second subindex and, in case of equality, according to their first subindex. $X_k$ is the biggest element, being the root.

The first production rule of $\mathbb A_k$ will be
$$
    X_k \rightarrow X_{c_1,d_1}\cdots X_{c_t,d_t}.
$$

Now let $X_{c,d}$ be a non-terminal symbol corresponding to a word $u^v$, let $a_1$ be the first letter of $v$, and write $v=a_1 v'$, so $u^v = (u^{a_1})^{v'}$. We know, from the relations in \cref{teoprespure1} and \cref{remarkrelations}, that the braid represented by $u^{a_1}$ can be written as $b_1\cdots b_s$, where each $b_i\in \mathcal F_k$ and $s\leq 9$. Therefore, the braid represented by $u^v$ can be written as $b_1^{v'}\cdots b_s^{v'}$, which we can encode as $X_{e_1,d-1}\cdots X_{e_s,d-1}$ for some integers $e_1,\ldots, e_s$. We thus add the following production rule to $\mathbb A_k$:
$$
    X_{c,d} \rightarrow X_{e_1,d-1}\cdots X_{e_s,d-1}.
$$
Adding these production rules for all non-terminal symbol $X_{c,d}$ ($d>0$) determines the compressed word $\mathbb A_k$.

\begin{proposition}\label{P:size_and_evaluation_with_boundary}
The size of $\mathbb A_k$ is smaller than $19(2g+p+n)m$, and the evaluation $ev(\mathbb A_k)$ represents $\alpha_k$.
\end{proposition}

\begin{proof}
The first production rule of $\mathbb A_k$ has length $t \leq m$. The length of the other production rules is at most 9, and there are as many as elements in $\mathcal A_k\backslash \{X_k\}$, that is, $2(j_k-1)|v_{i-1}|=2(2g+p+k-2)|v_{i-1}|< 2(2g+p+n)(m-1)$. Hence $|\mathbb A_k|\leq m+18(2g+p+n)(m-1)< 19(2g+p+n)m$.

By induction on $d$ (starting with $d=0$), we see that each $X_{c,d}$ evaluates to a word which represents the same braid as $u^v$ (where $u^v$ is the word corresponding to the symbol $X_{c,d}$). Hence, the evaluation of $X_k$ is a word in $\mathcal F_k$ which represents the same braid as $u_{i_1}^{v_{i_1}}\cdots u_{i_t}^{v_{i_t}}$, that is, $\alpha_k$.
\end{proof}

\begin{theorem}\label{teofinal}
Let $S$ be a connected surface of genus $g\geq 0$ with $p>0$ boundary components. Let $\alpha\in P_n(S)$ be a braid given as a word $w$ of length $m$ in the generators $\{A^{\pm 1}_{i,j}\}_{i<j}$. Then, for every $k=1,\ldots,n$, there is an algorithm of complexity $O((g+p+n)m)$ which produces a compressed word $\mathbb A_k$ of size at most $19(2g+p+n)m$ and terminal characters $\{A^{\pm 1}_{i,j_k}\}_{i<j_k}$, whose evaluation represents $\alpha_k$, the $k$-th factor of the combed normal form of $\alpha$.
\end{theorem}

\begin{proof}
This results follows from \cref{L:short_w_k_bar_with_boundary} and \cref{P:size_and_evaluation_with_boundary}, taking into account that computing the production rule corresponding to each $u^v$ just requires transcribing the relation in \cref{teoprespure1} and \cref{remarkrelations} corresponding to $u$ and to the first letter of $v$, so the complexity of the whole algorithm is proportional to the size of $\mathbb A_k$.
\end{proof}

\begin{example}
Consider the word $w = A_{1,4}A_{1,3}A_{2,4}^{-1}A_{1,2} \in P_4$. If we set $k=4$, we have $i_1=1$ and $v_{i_1}=A_{1,3}A_{1,2}$ which has length~2. The word $\overline w_4$, which is equal to $A_{14}^{A_{1,3}A_{1,2}}(A_{2,4}^{-1})^{A_{1,2}}$, can be codified as $X_{1,2}X_{-2,1}$. The production rules for $\mathbb A_4$ are:
$$
     X_4\rightarrow X_{1,2}X_{-2,1}
$$
$$
  \begin{array}{lll}
    X_{1,2}\rightarrow X_{1,1}X_{3,1}X_{1,1}X_{-3,1}X_{-1,1}    &  &   X_{-2,1}\rightarrow X_{1,0}X_{-2,0}X_{-1,0} \\ \\
    X_{1,1}\rightarrow X_{1,0}X_{2,0}X_{1,0}X_{-2,0}X_{-1,0}  &  & X_{3,1}\rightarrow X_{3,0} \\ \\
    X_{-1,1}\rightarrow X_{1,0}X_{2,0}X_{-1,0}X_{-2,0}X_{-1,0}  &  & X_{-3,1}\rightarrow X_{-3,0} \\ \\
  \end{array}
$$

The word $ev(\mathbb A_4)$ is $$X_{1,0}X_{2,0}X_{1,0}X_{-2,0}X_{-1,0}X_{3,0}X_{1,0}X_{2,0}X_{1,0}X_{-2,0}X_{-1,0}X_{-3,0}X_{1,0}X_{2,0}X_{-1,0}X_{-2,0}X_{-1,0},$$ corresponding to $$A_{1,4}A_{2,4}A_{1,4}A^{-1}_{2,4}A^{-1}_{1,4}A_{3,4}A_{1,4}A_{2,4}A_{1,4}A^{-1}_{2,4}A^{-1}_{1,4}A^{-1}_{3,4}A_{1,4}A_{2,4}A^{-1}_{1,4}A^{-1}_{2,4}A^{-1}_{1,4}.$$
\end{example}

Recall that the compressed braid combing explained throughout this section can be applied to the case when $S$ is a non-orientable surface with $p>0$ boundary components (one just needs to modify the production rules so they encode the appropiate relations given in \cite[Theorem 3]{GoncGuaschNonOrient} instead of the $(PR/ER)$-relations). Since the length of each of those relations is at most 9, the bounds given in \cref{P:size_and_evaluation_with_boundary} and \cref{teofinal} also hold for the non-orientable case.

\begin{corollary}[Word problem]
Let $S$ be a compact, connected surface of genus $g$ with $p>0$ boundary components. There exists an algorithm which, given two pure braids $\alpha_1, \alpha_2 \in P_n(S)$ represented by words $w_1, w_2$ of respective lengths $m_1,m_2$ in the generators of \cref{teoprespure1}, determines whether $\alpha_1=\alpha_2$ in polynomial time and space with respect to  $g, p, m_1$ and $m_2$.
\end{corollary}

\begin{proof}
The algorithm just applies compressed braid combing to $w_1$ and $w_2$, then transforms each compressed factor into a reduced compressed factor (\cref{T:reduce_SLP}), and then compares the reduced compressed factors to see whether they evaluate to the same reduced word (\cref{T:Plandowski}). All these steps can be done in polynomial time, as we have explained.
\end{proof}

\section{Braid combing on a closed surface}\label{sectioncombingclosed}

Throughout this section, assume that $S$ is a closed surface (compact, connected without boundary). In this case we cannot apply braid combing as in the previous section, since the exact sequence~\labelcref{ESnm} splits if and only if either $S$ is the sphere, the torus or the Klein bottle, or if $S$ is the projective plane with $n=3$ and $m=2$, or if $S$ is a surface of bigger genus and  $m=1$~\cite{GuaschiJuanPineda}. Moreover, $\pi_1(S)$ is no longer a free group, if $S$ is not the sphere.

One can nevertheless decompose $P_n(S)$ by using another instance of Sequence~\labelcref{ESnm}, the one in which $m=1$:
\begin{equation}~\label{ESn1}
   1 \rightarrow P_{n-1}(S\backslash \mathcal \{p_1\}) \stackrel{i_{n,1}}{\longrightarrow} P_n(S) \stackrel{p_{n,1}}{\longrightarrow} \pi_1(S) \rightarrow 1.
\end{equation}

This sequence is exact if $S$ is not the sphere or the projective plane, so we will exclude those two cases. In all other cases, the sequence splits~\cite{GoncGuasch,GuaschiJuanPineda}. Since we are assuming that $S$ has no boundary, we will use the generators of $P_n(S)$ defined in \cref{teoprespure2}.

One can then decompose $P_n(S)=\pi_1(S)\ltimes P_{n-1}(S\setminus \{p_1\})$. Then $S\setminus \{p_1\}$ can be treated as a surface with boundary, so $P_{n-1}(S\setminus \{p_1\})$ can be decomposed as a semi-direct product of free groups. In the case in which $S$ is orientable with genus $g>0$ we obtain a decomposition:
$$
    P_n(S)=\pi_1(S)\ltimes ((\cdots ((\mathbb F_{2g}\ltimes \mathbb F_{2g+1})\ltimes \mathbb F_{2g+2})\ltimes \cdots \mathbb F_{2g+n-3})\ltimes \mathbb F_{2g+n-2}).
$$
Combing a braid in $P_n(S)$, when $S$ is a closed orientable surface distinct from $\mathbb S^2$, means to find its normal form with respect to this group decomposition.

It is clear that, in order give an algorithm for braid combing, one needs to describe an explicit group section for the projection $p_{n,1}$. Notice that the generators of $\pi_1(S)$, say $a_1,\ldots,a_{2g}$, are naturally associated to $A_{1,j_1},A_{2,j_1}\ldots,A_{2g,j_1}$, where $j_1=2g+1$. But (contrary to the case with boundary) the map which sends $a_i$ to $A_{i,j_1}$ is not a group homomorphism: a section for $p_{n,1}$ must be defined otherwise.

In~\cite{GoncGuasch}, such a section is defined topologically, by using a retraction of the surface $S$, and allowing the distinguished points $p_2,\ldots,p_n$, to move along the retraction as the point $p_1$ performs the movement corresponding to some $a_i$. In~\cite{GoncGuasch}, this map from $\pi_1(S)$ to $P_n(S)$ is only described algebraically in the case of an orientable surface of genus 2.

We will now define a different group section, simpler than the one defined in~\cite{GoncGuasch} (although related to it), which will be explicitly given in terms of the generators of $P_n(S)$.

The generators of $P_n(S)$ are described in \cref{figgeneradoresPnSClosed}. These are analogous to the generators described in \cref{figgeneradoresPnS}, the only differences being that $S$ has no boundary (hence the second indices of the generators are shifted), and that we have placed the base points $p_1,\ldots,p_n$ in a different place to simplify the forthcoming figures.

\begin{figure}[ht]
\centering
\includegraphics[width = 12.2cm]{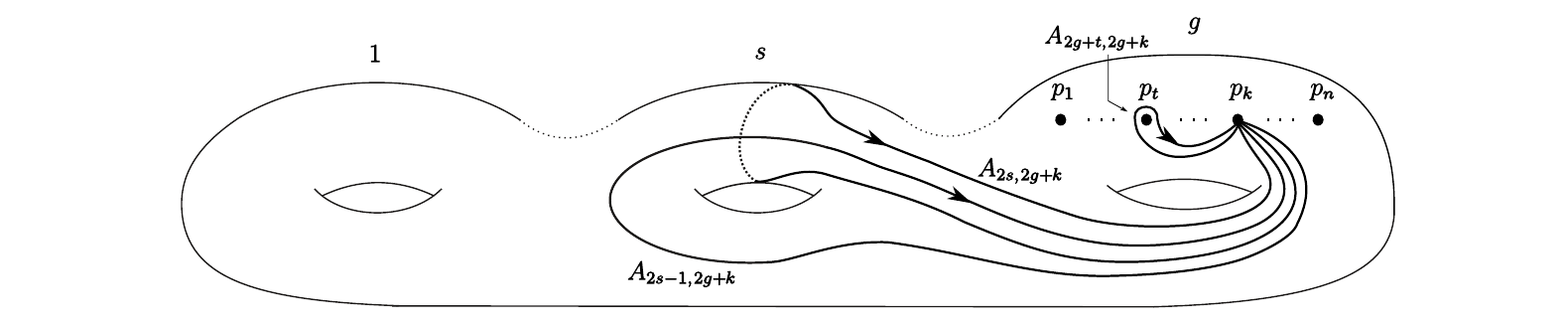}
\caption{\small{A geometric representation of the generators of $P_n(S)$ when $S$ is closed.}}
\label{figgeneradoresPnSClosed}
\end{figure}

Let us define, for $k=1,\ldots,n$, the braid $B_k=A_{2g,j_k}A_{j_1,j_k}A_{j_2,j_k}\cdots A_{j_{k-1},j_k}$, where $j_t=2g+t$ for $t=1,\ldots,n$. See \cref{figB_k} for a picture of $B_k$ as a product of generators $(a)$, and also in a simpler geometric way $(b)$.

\begin{figure}[ht]
\centering
\includegraphics[width = 12.2cm]{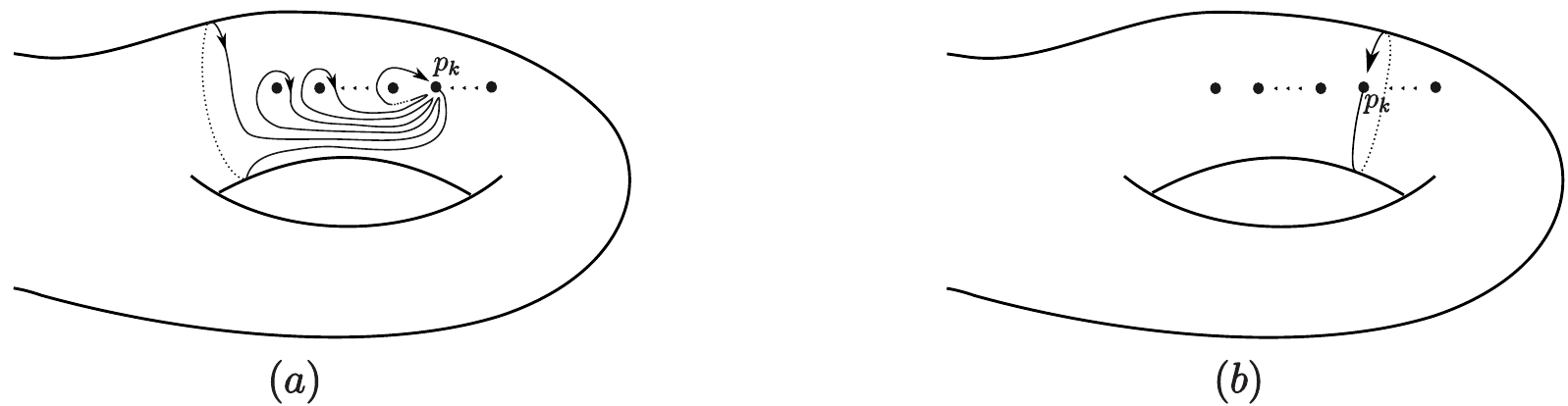}
\caption{\small{The braid $B_k=A_{2g,j_k}A_{j_1,j_k}A_{j_2,j_k}\cdots A_{j_{k-1},j_k}$.}}
\label{figB_k}
\end{figure}

Using the geometric representation of the braid $B_k$ given in \cref{figB_k}, it is clear that $B_i$ and $B_j$ commute, for every $1\leq i, j\leq n$. We will be particularly interested in the braid $B_1B_2\cdots B_n$ which can be seen in \cref{figDisjuntos_B_k}.

\begin{figure}[ht]
\centering
\includegraphics[width = 12.2cm]{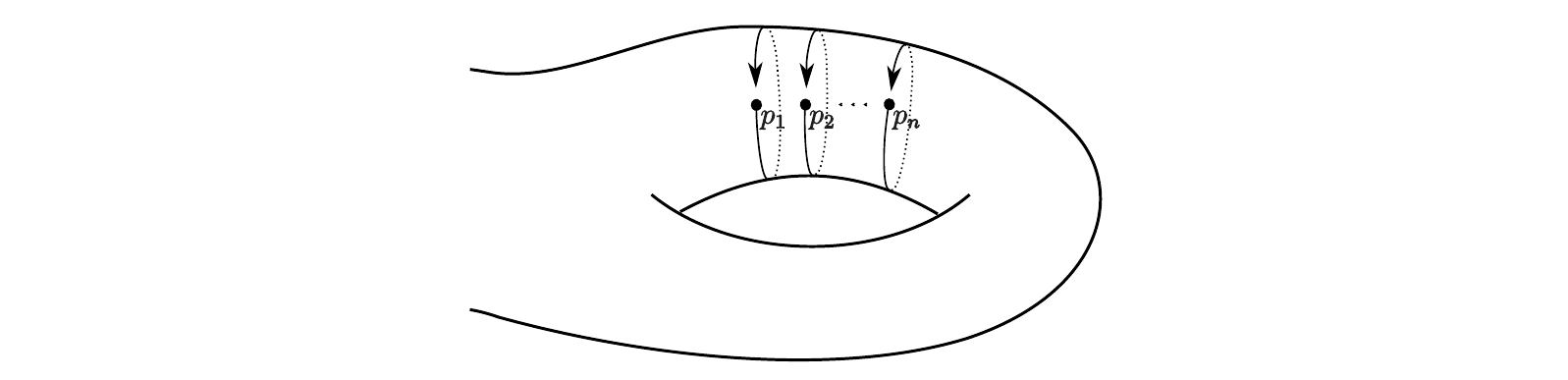}
\caption{\small{The braid $B_1B_2\cdots B_n$.}}
\label{figDisjuntos_B_k}
\end{figure}

In \cref{figpelicula} we can see the braid $(B_1B_2\cdots B_n)^{-1}A_{2g,j_1}(B_1B_2\cdots B_n)A^{-1}_{2g,j_1}$ in picture $(a)$, which is smoothly transformed into the braid in picture $(d)$. It is a classical exercise to see how to express the path in picture $(d)$ as a product of generators of the fundamental group of $S$. In our case, this allows to express that braid as a product of the generators of $P_n(S)$, as follows:
$$
   (B_1B_2\cdots B_n)^{-1}A_{2g,j_1}(B_1B_2\cdots B_n)A^{-1}_{2g,j_1} = [A_{1,j_1},A^{-1}_{2,j_1}]\cdots [A_{2g-3,j_1},A^{-1}_{2g-2}].
$$

\begin{figure}[ht]
\centering
\includegraphics[width = 12.2cm]{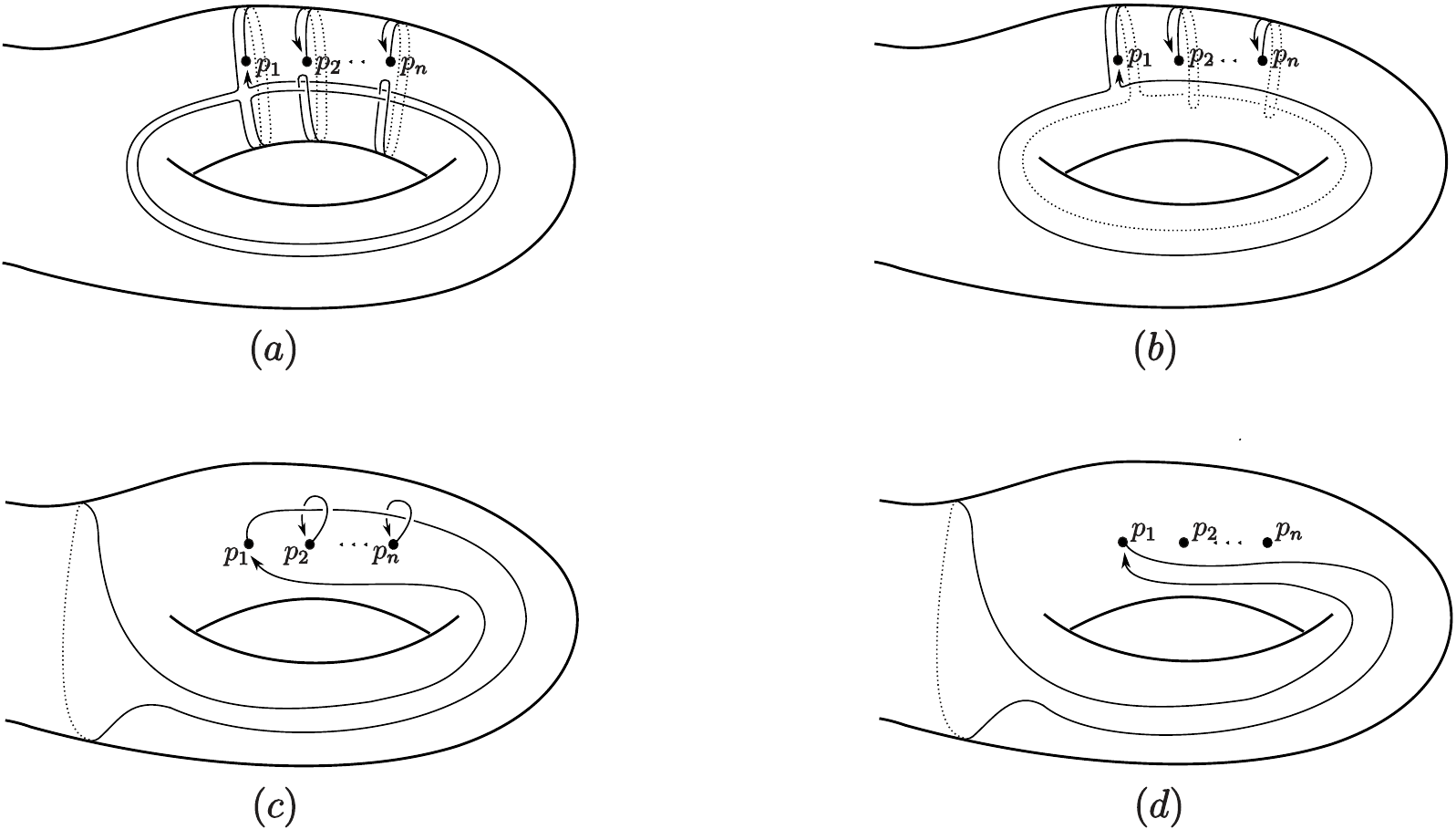}
\caption{\small{The braid $(B_1B_2\cdots B_n)^{-1}A_{2g,j_1}(B_1B_2\cdots B_n)A^{-1}_{2g,j_1}$.}}
\label{figpelicula}
\end{figure}

It follows that, in $P_n(S)$, one has
$$
   [(B_1B_2\cdots B_n)^{-1}, A_{2g,j_1}] [A^{-1}_{2g-2}, A_{2g-3,j_1}]\cdots [A^{-1}_{2,j_1}, A_{1,j_1}]=1.
$$
Now recall that the fundamental group of $S$ has the following presentation:
$$
   \pi_1(S)=\left\langle a_1,\ldots,a_{2g};\ [a^{-1}_{2g}, a_{2g-1}][a^{-1}_{2g-1}, a_{2g-2}]\cdots [a^{-1}_2, a_1]=1\right\rangle,
$$
where $a_i=p_{n,1}(A_{i,j_1})$ for $i=1,\ldots, 2g$. If we notice that $p_{n,1}(B_1B_2\cdots B_n)=a_{2g}$, the following result is immediately obtained:

\begin{theorem}\label{T:section}
Let $S$ be an orientable closed surface of genus $g\geq 1$. The map $s:\ \pi_1(S) \rightarrow P_n(S)$ which sends $a_i$ to $A_{i,j_1}$ for $i=1,\ldots, 2g-1$, and $a_{2g}$ to $B_1\cdots B_n$, is a group section of the projection $p_{n,1}$ of the short exact sequence~\labelcref{ESn1}.
\end{theorem}

\begin{remark}
To our knowledge, the above result gives the first known explicit algebraic section of $p_{n,1}$, when $g>2$. Moreover, we can also give an explicit algebraic definition of the group section described geometrically in~\cite{GoncGuasch}: It is the map $\rho: \pi_1(S)\rightarrow P_n(S)$ such that $\rho(a_i)=A_{i,j_1}$ when $i$ is odd, and $\rho(a_i)=A_{i,j_1}B_2\cdots B_n$ when $n$ is even. The proof that $\rho$ is a group section is similar to the one we did for $s$. We used $s$ instead of $\rho$ as it is an algebraically simpler section.
\end{remark}

Now we can comb a braid in a closed orientable surface $S$ with genus $g>0$, using the above section (which allows to compute the normal form with respect to the decomposition $P_n(S)=\pi_1(S) \ltimes P_{n-1}(S\backslash\{p_1\})$), and then applying the combing procedure of \cref{sectioncombing} to the second factor, obtaining the normal form with respect to the decomposition:

$$
    P_n(S)=\pi_1(S)\ltimes ((\cdots ((\mathbb F_{2g}\ltimes \mathbb F_{2g+1})\ltimes \mathbb F_{2g+2})\ltimes \cdots \mathbb F_{2g+n-3})\ltimes \mathbb F_{2g+n-2}).
$$

There is one detail to be taken into account. When we decompose $P_n(S)=\pi_1(S) \ltimes P_{n-1}(S\backslash\{p_1\})$, the group $P_{n-1}(S\backslash\{p_1\})$ is considered as a subgroup of $P_n(S)$ (formed by the braids in which the first strand is trivial). In other words, the generators of this group are the braids of the form $A_{i,j_k}$, where $i\leq j_k=2g+k$ and $k>1$. But then we consider the group $P_{n-1}(S\backslash\{p_1\})$ as a braid group of a surface with boundary, $S'$, which is obtained by removing a small neighborhood of $p_1$.

The generators of $P_{n-1}(S')$ are shown in \cref{figgeneradoresPnS}, where the only boundary component is placed on the right hand side of the picture. We can express any word in the generators of $P_{n-1}(S\backslash\{p_1\})$ as a word in the generators of $P_{n-1}(S')$ thanks to the isomorphism $f: \ P_{n-1}(S\backslash\{p_1\}) \rightarrow P_{n-1}(S')$ defined as follows:
$$
   f(A_{i,j_k})=\left\{\begin{array}{ll}
     A_{i,j_{k-1}} & \mbox{if } i\leq 2g, \\ \\
   \displaystyle \left[A_{2g,j_{k-1}}^{-1}, A_{2g-1,j_{k-1}}\right]\cdots \left[A_{2,j_{k-1}}^{-1},A_{1,j_{k-1}}\right]\left(\prod_{t=1}^{k-2}A_{j_{t},j_{k-1}} \prod_{t=k}^{n-1}A_{j_{k-1},j_{t}}\right)^{-1} & \mbox{if } i=2g+1, \\ \\
     A_{i-1,j_{k-1}} & \mbox{if } 2g+1 < i <j_k.
   \end{array}\right.
$$
These formulae are obtained by interpreting the generators of $P_{n-1}(S\backslash\{p_1\})$ as points moving in the surface $S'$, in which the point $p_1$ has been transformed into a boundary component (and moved to the right hand side, like in \cref{figgeneradoresPnS}). All interpretations are straightforward, except the generator $A_{2g+1,j_k}=A_{j_1,j_k}$. In $P_{n-1}(S\backslash\{p_1\})$, this generator corresponds to a movement of the puncture $p_k$ around the puncture $p_1$. In $P_{n-1}(S')$, however, there is no generator corresponding to a puncture moving around the last (and only) boundary of $S'$. We must then apply the relation (TR) of \cref{teoprespure2}, to express $A_{j_1,j_k}$ as a product of other generators, which are then mapped to $P_{n-1}(S')$ as expressed in the above equation. Once we have applied the map $f$, we can comb the resulting braid in $P_{n-1}(S')$ as it was explained in \cref{sectioncombing}.

Now we will explain why we cannot apply the techniques in \cref{sectioncompressed} to comb a braid in a closed surface. The idea of combing, as it was done in \cref{sectioncombing}, is to {\it move to the left} the generators with smaller second index, which act by conjugation on the generators with bigger second index. If the surface $S$ has boundary, the action of a generator $A_{i,j}^{\pm 1}$ on a generator $A_{r,s}$, produces a word in which all letters have second index $s$. Hence, the $k$-th final factor of the combed braid only depends on the letters of the original word having second index $j_k$, and on the letters of bigger second subindex which act on them. This is why we can easily determine the compressed word associated to the $k$-th factor.

If the surface $S$ is closed, however, the action of a generator $A_{i,j}^{\pm 1}$ on a generator $A_{r,s}$, does not necessarily produce a word in which all letters have second index $s$, due to the necessity of applying the map $f$. As an example, consider the relation (ER1) with $j=j_1$:
$$
  A^{-1}_{r+1,j_1}A_{r,s}A_{r+1,j_1} = A_{r,s}A_{r+1,s}A^{-1}_{j_1,s}A^{-1}_{r+1,s}.
$$
All letters in the resulting word seem to have the same second subindex, but when we apply the map $f$ to see the braid in $P_{n-1}(S')$, the letter $A^{-1}_{j_1,s}$ must be replaced by a word whose letters have second subindex going from $s-1$ to $n-1$. This fact does not permit to obtain the factors of a combed braid as compressed words, as was done in \cref{sectioncombing}. A different approach should therefore be used, in order to find a polynomial solution to the word problem of braid groups on closed surfaces.

Nevertheless, the algebraic description of the section $s$ in \cref{T:section} allows to perform (non-compressed) braid combing in the classical way, as explained in this section. This was not possible before, due to the lack of an algebraically explicit section. This procedure of combing a braid in a closed surface is, however, exponential.


\vspace{0.5cm}

\begin{minipage}[l]{10cm}
\noindent \textbf{Juan Gonz\'alez-Meneses} \\     
Dpto. \'Algebra \\
Facultad de Matem\'aticas \\ 
Instituto de Matem\'aticas (IMUS)\\
Universidad de Sevilla\\
Avda. Reina Mercedes s/n\\
41012 Sevilla (SPAIN) \\
{\tt meneses@us.es}\\ 
\end{minipage}
\begin{minipage}[l]{5cm}
\noindent \textbf{Marithania Silvero} \\
Institute of Mathematics \\
Polish Academy of Sciences\\
ul. \'Sniadeckich, 8 \\
00-656 Warsaw (POLAND) \\
{\tt marithania@us.es} \\
\\
\\
\end{minipage}

\end{document}